\documentclass[12pt,leqno,twoside]{amsart}
\usepackage{amssymb,amsfonts,amsmath,amsthm, amstext}
\usepackage{latexsym}
\usepackage{verbatim}
\usepackage{graphicx}

\usepackage[latin1]{inputenc}
\usepackage[T1]{fontenc}
\usepackage[colorlinks=true, pdfstartview=FitV, linkcolor=blue, citecolor=blue, urlcolor=blue]{hyperref}
\usepackage{amstext,amsmath,amscd, bezier,indentfirst,amsthm,amsgen,enumerate}
\usepackage[all,knot,arc,import,poly]{xy}
\usepackage{amsfonts,color, soul}  
\usepackage{graphicx}
\usepackage{srcltx}
\usepackage{enumitem}
\usepackage{epsfig}

\usepackage{tikz}
\usepackage{mathrsfs}
\usepackage[margin=1in]{geometry}
\usepackage[normalem]{ulem}

\newtheorem{theorem}{Theorem}[section]
\newtheorem{corollary}[theorem]{Corollary}
\newtheorem{lemma}[theorem]{Lemma}
\newtheorem{proposition}[theorem]{Proposition}

\theoremstyle{definition}
\newtheorem{definition}[theorem]{Definition}
\theoremstyle{remark}
\newtheorem{remark}[theorem]{\sc Remark}
\theoremstyle{remark}
\newtheorem{notations}[theorem]{\sc Notations}
\theoremstyle{remark}
\newtheorem{example}[theorem]{\sc Example}
\theoremstyle{remark}

\theoremstyle{remark}

\newtheorem{lem/defi}[thm]{Lemma/Definition}
\newtheorem{defi/lem}[thm]{Definition/Lemma}
\newtheorem{prop/defi}[thm]{Proposition/Definition}

\theoremstyle{definition}

\theoremstyle{remark}

\renewcommand{\Box}{\square}    

\newcommand{\gr}{{\mathrm{gr}}}

\newcommand{\con}{\mathrm{cont}}
\newcommand{\hot}{\mathrm{h.o.t.}}
\newcommand{\jac}{{\mathrm{Jac} \ }}

\renewcommand{\top}{{\mathrm{top}}}
\newcommand{\id}{{\mathrm{id}}}

\newcommand{\im}{\mathop{\rm{Im}}\nolimits}
\newcommand{\mult}{{\rm{mult}}}

\newcommand{\ord}{{\mathrm{ord}}}

\newcommand{\grad}{\mathop{\rm{grad}}\nolimits}
\newcommand{\Grad}{\mathop{\rm{grad}}\nolimits}

\newcommand{\enr}{{\rm{enr}}}
\newcommand{\Horn}{{\rm{Horn}}}
\newcommand{\e}{\varepsilon}
\newcommand{\fin}{\hspace*{\fill}$\Box$}

\newcommand{\oa}{\omega}
\newcommand{\ga}{\gamma}
\newcommand{\al}{\alpha}
\newcommand{\GC}{\mathcal{GC}}

\newcommand{\mcv}{\mathcal{V}}

\newcommand{\mer}{\mathscr M}

\newcommand{\ra}{\rightarrow}
\newcommand{\mt}{\mapsto}
\newcommand{\be}{\begin{equation}}
\newcommand{\ee}{\end{equation}}

\newcommand{\beqn}{\begin{eqnarray*}}
\newcommand{\eeqn}{\end{eqnarray*}}

\newcommand{\m}{\medskip}

\newcommand{\cD}{{\mathcal D}}

\newcommand{\I}{\mathcal{I}}

\newcommand{\cL}{{\mathcal L}}
\newcommand{\PP}{\mathcal{NP}}

\newcommand{\V}{\mathcal{V}}

\newcommand{\bR}{{\mathbb R}}
\newcommand{\bC}{{\mathbb C}}

\newcommand{\F}{\mathbb{F}}
\newcommand{\bP}{{\mathbb P}}
\newcommand{\bN}{{\mathbb N}}

\newcommand{\bQ}{{\mathbb Q}}

\newcommand{\C}{{\mathbb C}}

\newcommand{\Q}{\mathbb{Q}}
\newcommand{\Z}{\mathbb{Z}}
\newcommand{\R}{\mathbb{R}}

\begin{document}

\title[Concentration of curvature and Lipschitz invariants]{Concentration of curvature and Lipschitz invariants of holomorphic functions of two variables}

\author{Lauren\c tiu P\u aunescu}
\address{School of Mathematics and Statistics, University of Sydney,
  Sydney, NSW, 2006, Australia.}
\email{laurent@maths.usyd.edu.au}

\author{Mihai Tib\u ar}
\address{Math\' ematiques, UMR 8524 CNRS,
Universit\'e de Lille, \  59655 Villeneuve d'Ascq, France.}
\email{mihai-marius.tibar@univ-lille.fr}

\thanks{The authors acknowledge the support of the Labex CEMPI
(ANR-11-LABX-0007-01). They  thank the CIRM at Luminy and the MFO at Oberwolfach for hosting them in ``Research in Pairs'' programs during the preparation of this manuscript.  }

\begin{abstract}
 By  combining analytic and geometric viewpoints on the concentration of the curvature
 of the Milnor fibre,  we prove that Lipschitz homeomorphisms preserve the zones of multi-scale curvature concentration as well as the gradient canyon structure  of holomorphic functions of two variables. This yields the first new Lipschitz invariants
after those discovered by Henry and Parusi\'nski in 2003.
   

\end{abstract}

\maketitle
\setcounter{section}{0}

\section{Introduction}\label{s:intro}
 


For two variables holomorphic function germs,  the first bi-Lipschitz invariants, different from the topological invariants, were found around 2003 by Henry and Parusi\' nski \cite{HP, HP2} who showed that there are moduli.


In a different stream, Garcia Barroso and Teissier \cite{GT} had shown that the total curvature of the Milnor fibre concentrates in a multi-scale manner along a certain truncation of the generic polar curve of the two variable holomorphic function.

More recently, Kuo, Koike and P\u aunescu \cite{kkp, KKP} studied the bumps of curvature on the Milnor fibre by using the gradient canyons as key devices.

By using these three complementary viewpoints, all of which gravitate around the geometric and analytic properties of the polar curves, we show here that Lipschitz homeomorphisms preserve the gradient canyon structure of holomorphic functions $f : (\bC^{2}, 0)\to (\bC, 0)$, cf Theorem \ref{t:main2}. More precisely, the gradient canyons, together with their clustering and contact orders,  are Lipschitz invariants of holomorphic functions of two variables.
They complement the Henry-Parusi\' nski  continuous invariants \cite{HP, HP2}, as demonstrated by Example \ref{e:2}.  
 
 
 

In order to state our main result, we give an account of the canyon data  and send to the next sections for the details.
Let $\gamma_{*}$ denote some \emph{polar} of  $f = f(x,y) : (\bC^{2}, 0) \to (\bC, 0)$,  i.e. an irreducible curve which is a solution of the equation $f_{x}=0$. 
We consider some Newton-Puiseux parametrization of  it, i.e. of the form $\alpha(y) = (\gamma(y) , y)$,
which can be obtained by starting from a holomorphic parametrization $\alpha : (\bC,0) \to (\bC^{2}, 0)$, $\alpha(t) = (\alpha_{1}(t), \alpha_{2}(t))$ with  $\ord_{t}\alpha_{2} \le \ord_{t}\alpha_{1}$,
and then making the change of parameter  $\alpha(y) = (\gamma(y) , y)$ with $y = \alpha_{2}(t)$.  Then $m:= \mult(\gamma_{*}) := \min  \ord_{t}\alpha_{2}$  (minimum over all parametrisations)  is the 
\emph{multiplicity}  of the polar $\gamma_{*}$; it is also equal to the total number of parametrizations of  $\gamma_{*}$ of order $m$ which are conjugate. 

Let $d_{\gr}(\gamma)$ be the degree for which the order $\ord_y(\|\Grad f(\gamma(y),y)\|)$ of the gradient is stabilized, see \eqref{GG}, 
and let $\GC(\gamma_{*})$ be \emph{the gradient canyon of $\gamma$} (Definition \ref{d:canyon}).  Such a canyon contains one or more polars with the same  canyon degree $d_{\gr}(\gamma)$. The \emph{multiplicity of the canyon} $\mult(\GC(\gamma_{*}))$ is the sum of the multiplicities of its polars. 

\smallskip

Let us point out that the gradient canyons and their degrees are not topological invariants, see  Example \ref{e:2}. 

\medskip

While analytic maps do not preserve polars, we prove the analytic invariance of the canyons  as a preamble for the definition of our new bi-Lipschitz invariants, Theorem \ref{t:main1}: 
\emph{If $f = g\circ \varphi$ with $\varphi$ analytic bi-Lipschitz, then $\varphi$ transforms canyons into canyons by preserving their degrees and multiplicities.}
It follows that the map $\varphi$ establishes a bijection between the canyons  of $f$ and those of $g$
such that the degrees  $d_{\gr}(\gamma_{*})$ and the multiplicities $\mult(\GC(\gamma_{*}))$ are the same. 


\smallskip

When we drop the analyticity assumption of the bi-Lipschitz map $\varphi$, the  perspectives are challenging since not only that polar curves are not sent to polar curves, but we cannot prove anymore that gradient canyons are sent to gradient canyons.
Up to now, the only result in full generality has been obtained by Henry and Parusi\'nski \cite{HP, HP2}, namely the authors have 
found  that the leading coefficient in the expansion \eqref{eq:h-coefficient}, modulo an equivalence relation, is a
bi-Lipschitz invariant. More than that,  Henry and Parusi\'nski showed in \cite{HP, HP2} that  a certain zone in the Milnor fibre,  which is  characterised by
the higher order of the change of the gradient, is preserved by bi-Lipschitz homeomorphisms.

Our new  bi-Lipschitz invariants extend in a certain sense the discrete set of 
topological invariants of plane curves, but they refer to the branches of the polar curve instead of the branches
of the curve $\{f=0\}$. Our clustering description of the polar curves and their associated zones refines in a multi-scale manner the  Henry-Parusi\'nski zone. As of comparing our invariants to the Henry-Parusi\'nski Lipschitz continuous invariants,  Example \ref{e:2} shows that they are complementary.

We  establish in \S \ref{s:genpolars} a faithful correspondence between the concentration of curvature invariants coming from 
Garcia Barroso and Teissier's geometric study \cite{GT} and those coming from the Koike, Kuo and P\u aunescu analytic study \cite{KKP}, in particular we prove (Theorem \ref{t:tau}):
\emph{the contact degree $d_{\gamma(\tau)}$ and the gradient canyon $\GC(\gamma(\tau))$ do not depend of the direction $\tau$ of the polar $\gamma(\tau)$, for generic $\tau$.}
This result also contributes to the proof of our main  results Theorems \ref{t:main0} and  \ref{t:main2} in Section \ref{s:disks}, of which we give  a brief account in the following.

\smallskip
 Let $f = g \circ \varphi$ with $\varphi$ a bi-Lipschitz homeomorphism. Even if the image by  $\varphi$ of a gradient canyon is not anymore a gradient canyon like in the analytic category, our key result 
Theorem \ref{t:main0} says that:
 \emph{the bi-Lipschitz map $\varphi$ establishes a bijection between the canyon disks of $f$ and the canyon disks of $g$ by preserving the canyon degree},  
  where
the \emph{canyon disks}  are defined as the intersections of the \emph{horn domains} \eqref{eq:hd} with the Milnor fibre.
 
We are therefore in position to prove that $\varphi$ induces a bijection between the gradient canyons of $f$ and those of $g$, and  moreover, that there are \emph{clusters of canyons} of $f$ which correspond by $\varphi$ to similar clusters of $g$. Such clusters are defined in terms of orders of contact (i.e. certain rational integers) which are themselves bi-Lipschitz invariants.
Our main result, Theorem \ref{t:main2},  states minutiously this correspondence, and we send to Section \ref{s:disks} for its formulation  and the preparatory definitions.

\begin{example} \label{e:1} The function germ $f :=z^{4}+  z^{2}w^{2}+ w^{4}$ has 3 polars with canyon degrees $d=1$ belonging to a single canyon of multiplicity 3.  One can show\footnote{by using Kuo's trivialising vector field  in the family $f_{t}=z^{4}+  tz^{2}w^{2}+ w^{4}$ which is homogeneous of degree 4.} that $f$ is bi-Lipschitz equivalent to the function germ  $g :=z^{4}+ w^{4}$ which has a single polar, its canyon has degree $d=1$ and multiplicity 3. The canyon degree is  indeed a bi-Lipschitz invariant, according to our Theorem \ref{t:main0}, but the number of polars in the canyon is not invariant.
  \end{example}
  

\begin{example} \label{e:2}
For instance, consider the function germs $f(x,y)=x^3+y^{12}$ and $g(x,y)=x^3+y^{12}+ x^2y^5$ which are topologically equivalent\footnote{since $f_{t}(x,y)=x^3+y^{12}+ tx^2y^5$ is a topologically trivial family.}.
Then $g$ has two disjoint canyons corresponding to the two distinct polars, both having degree  $d=6$, whereas
$f$ has only one double polar with canyon degree $d=\frac{11}{2}$, hence only one canyon. According to our Theorem  \ref{t:main0}, these two function germs are not Lipschitz equivalent. Nevertheless they have the same Henry-Parusinski  invariants \cite{HP, HP2}.
\end{example}

\section{Gradient canyons}\label{structures}
We recall from  \cite{kuo-pau2} and \cite{KKP} some of the definitions and results that we shall use.

One calls \textit{holomorphic arc} the image  $\alpha_* :=\im(\tilde\alpha)$ of an irreducible plane curve germ:
\[ 
\tilde\al: (\bC,0)\longrightarrow (\bC^2,0), \quad \tilde\al(t)=(z(t),w(t)).
\] 
 It has a unique complex tangent line $T(\alpha_*)$ at $0$, considered as a point in the projective line, i.e.  $T(\alpha_*)\in\bC P^1$.  The total space of  holomorphic arcs was called ``enriched Riemann sphere'' in \cite{kuo-pau2, KKP}.

The classical Newton-Puiseux Theorem asserts that the field $\F$
of convergent fractional power series in $y$ is
algebraically closed, see e.g.  \cite{walker}, \cite{wall}. A non-zero element of $\F$ is a (finite or infinite)
convergent series with positive rational exponents: 
\[
\alpha(y)=a_0y^{n_0/N}+\cdots  +a_iy^{n_i/N}+\cdots ,\quad
n_0<n_1<\cdots,
\]

where $0\ne a_i\in \bC$, $N, n_i \in \bN$, $N>0$, with
$\mathrm{gcd}(N, n_0, n_1, ...)=1$,   $\limsup_{i}|a_i|^{\frac{1}{n_i}}<\infty$.

   The \textit{conjugates} of $\alpha$ are
\[ \alpha_{conj}^{(k)}(y) :=\sum a_i \theta^{kn_i}y^{n_i/N},\mbox{ where }
0\leq k\leq N-1  \mbox{ and } \theta :=e^\frac{2\pi \sqrt{-1}}{N}.
\]
The \textit{order} of $\alpha$ is 
\begin{eqnarray*}& \ord(\alpha) :=\ord_y(\alpha) =\frac{n_0}{N} \text{ if }\;\, \alpha\ne 0 \mbox{ and } \ord(\alpha) :=\infty  \text{ if } \alpha=0,&
\end{eqnarray*}
 and  $m_{puiseux}(\alpha) :=N$ is the \textit{Puiseux multiplicity} of $\alpha$.

For any $\alpha\in \F_1:=\{\alpha|\ord_y(\alpha)\geq 1\}$,   the map germ
\[ 
\label{para}
\tilde\alpha: (\bC,0)\longrightarrow (\bC^2,0), \quad t\mt (\alpha(t^N),t^N),\quad N :=m_{puiseux}(\alpha),
\] 
is holomorphic and the holomorphic arc $\alpha_*$ is then well defined.

\smallskip
\noindent
 One defines several subspaces of holomorphic arcs, as follows.

 For some fixed $\alpha_* =\im(\tilde\alpha)$,   one defines:
\begin{equation}\label{ID}
\mathcal{D}^{(e)}(\alpha_*;\rho):=\{\beta_*\,\mid \beta(y)=[J^{(e)}(\alpha)(y)+cy^e]+\hot ,\,\;|c|\leq \rho\},
\end{equation}
where $1\leq e<\infty$, $\rho\geq 0$, and where $J^{(e)}(\alpha)(y)$
is the $e$-\textit{jet} of $\alpha$ and ``$\hot$'' means as usual ``higher order terms''. Moreover, one defines:

\begin{equation}\label{IL}
\begin{split}\mathcal{L}^{(e)}(\alpha_*)&:=\mathcal{D}^{(e)}(\alpha_*;\infty) :=\cup_{0<\rho<\infty}\mathcal{D}^{(e)}(\alpha_*;\rho)\\&=
\{\beta_* \mid \beta(y)=[J^{(e)}(\alpha)(y)+cy^e]+\hot,\; |c| \in \bR\}.
\end{split}
\end{equation}



%
Note that in the above  definitions \eqref{ID} and \eqref{IL}, the parameter $\alpha\in \F_{1}$ runs over all its conjugates.

\m

%



Consider the Newton-Puiseux factorizations:
\begin{eqnarray}\label{ff}
f(x,y)=u\cdot \prod _{i=1}^m(x-\zeta_i(y)), \quad f_x(x,y)=v\cdot \prod_{j=1}^{m-1}(x-\ga_j(y)),
\end{eqnarray} 
where $\zeta_i$, $\gamma_j\in \F_1$ and $u,v$ are units.   Note that all conjugates of roots are also roots.
We call \textit{polar} any such root $\gamma_j$, as well as its geometric representation  $\gamma_{j*}$. 

If a polar $\gamma$ is also a root of $f$, i.e. $f(\gamma(y),y)\equiv 0$, then it is a multiple root of $f$. 

\m


From the Chain Rule it follows:
\begin{equation}\label{EMF}
f_x(\alpha(y),y)\equiv f_y(\alpha(y),y)\equiv 0 \implies f(\alpha(y),y)\equiv 0
\end{equation}
for any $\alpha\in \F_1$.
Let us fix a  polar $\gamma$ with $f(\gamma(y),y)\not \equiv 0$. By (\ref{EMF}), $\gamma$ is not a common Newton-Puiseux root of $f_z$ and $f_y$. 
If $q$ is sufficiently large, then one has the equality:
\begin{equation}\label{GG}
\ord_y(\|\Grad f(\gamma(y),y)\|)=\ord_y(\|\Grad f(\gamma(y)+uy^q,y)\|), \; \forall\, u\in \bC.
\end{equation}

\begin{definition}\label{d:canyon}
The \textit{gradient degree} $d_{\gr}(\gamma)$ is the smallest number $q$ such that (\ref{GG}) holds for generic $u\in \bC$.
In the case $f(\gamma(y),y)\equiv 0$, one sets $d_{\gr}(\gamma) :=\infty$.
\end{definition}

It turns out that the gradient degree $d_{\gr}(\gamma)$ is rational since it is a co-slope in 
a Newton polygon, see Lemma \ref{LH}(i) for $\alpha := \gamma$. It can also be interpreted as a \L ojasiewicz exponent.

\begin{definition}
Let $\gamma$ be a polar of gradient degree $d :=d_{\gr}(\gamma)$, $1\leq d \leq \infty$. The \textit{gradient canyon of $\gamma_*$} is by definition
\[ \GC(\gamma_*) :=\mathcal{L}^{(d)}(\gamma_*).
\]

One calls $d_{\gr}(\gamma_*) :=d_{\gr}(\gamma)$ the \textit{gradient degree of $\gamma_{*}$}, or the \textit{degree of $\GC(\gamma_*)$}.

One says that $\mathcal{GC}(\gamma_*)$ is \textit{minimal} if $d_{\gr}(\gamma_*)<\infty$ and if, for any polar $\gamma_i$ of finite degree, 
the inclusion $ \mathcal{GC}(\gamma_{i*}) \subseteq\mathcal{GC}(\gamma_*)$ implies  the equality $\mathcal{GC}(\gamma_{i*})= \mathcal{GC}(\gamma_*)$.
\end{definition}

\begin{definition}
The \textit{multiplicity} of the gradient canyon $\GC(\gamma_*)$ is defined as:
\begin{equation}\label{multi} 
\mult(\GC(\gamma_*)) :=\sharp \{j\,\mid 1\leq j \leq m-1,\; \gamma_{j*}\in\mathcal{GC}(\gamma_*)\},
\end{equation}
where $m$ is as in \eqref{ff}.
\end{definition}
\m 
Up to some generic unitary change of coordinates, one has the following presentation: 
\begin{equation}\label{mini} 
f(x,y) :=f_m(x,y)+f_{m+1}(x,y)+\hot,\end{equation}
where $f_k$ denotes a homogeneous  $k$-form,  with $f_m(1,0)\ne 0$  and  $m=\ord(f)$.


The initial form $f_m(x,y)$  factors as:
\begin{equation}\label{HG}f_m(x,y)=c(x-x_1y)^{m_1}\cdots (x-x_ry)^{m_r},\;\, m_i\geq 1, \;\, x_i\ne x_j \;\,\text{if}\;\,i\ne j,
\end{equation} 
and $1\leq r\leq m$, $m=m_1+\cdots +m_r$, $c\ne 0$. 

\m
We have that  $f_m(x,y)$ is degenerate if and only if $r<m$.
The following useful result sheds more light over the landscape of gradient canyons:

\begin{theorem}\label{thmB} \cite[Theorem B]{KKP}
Any gradient canyon of degree $1<d_{\gr}<\infty$ is a minimal canyon. 
The  canyons of degrees $1<d_{\gr}\leq \infty$ are mutually disjoint. 
 
There are exactly $r-1$ polars of gradient degree $1$, counting multiplicities,  and they belong to the unique gradient canyon of degree 1, denoted by $\bC_{\enr}$. 

If $1<r\leq m$,  then $\bC_{\enr}$ is minimal if and only if $f(z,w)$ has precisely $r$ distinct roots $\zeta_i$ in (\ref{ff}).
In particular, if $f_m(x,y)$ is non-degenerate then $\bC_{\enr}$ is minimal. \fin
\end{theorem}

\m

\noindent 
\textbf{The horn, the partial Milnor number, and the total curvature of a gradient canyon.}\

A well-known formula to compute the Milnor number $\mu_f$  is the following:

\begin{equation}\label{MWK}
\mu_f=\sum_{j=1}^{m-1}\left[ \ord_y(f(\ga_j(y),y))-1\right] , 
\end{equation}
where the sum runs over all $\ga_j$, i.e. over all polars and their conjugates.

\m 

One defines the \emph{Milnor number of $f$ on a gradient canyon} $\mathcal{GC}(\ga_*)$ with $d_{\gr}(\ga_*)<\infty$, as:
\begin{eqnarray}\label{M}
&\mu_f(\mathcal{GC}(\ga_*))\!:=\sum_j[\ord_y(f(\ga_j(y),y))-1],&
\end{eqnarray}
where the sum is taken over all $j$, $1\leq j\leq m-1$, such that 
$\ga_{j*}\in\mathcal{GC}(\ga_*)$.

From (\ref{M}) and (\ref{multi})  one has:
\begin{eqnarray*}&
\mu_f(\GC(\ga_*))+\mult(\GC(\ga_*))=\sum_j\ord_y(f(\ga_j(y),y)),&
\end{eqnarray*}
where $j$ runs like  in the sum of  (\ref{M}).
%

Consider 
$\mathcal{D}^{(e)}(\al_*;\rho)$ as in \eqref{ID}, of finite order $ e \geq1$ and finite radius $\rho >0 $, and
 a compact ball
$B(0;\eta)\!:=\{(x,y)\in \C^2\,\mid \sqrt{|x|^2+|y|^2}\leq\eta\}$ with small enough $\eta>0$ (usually we consider a Milnor ball of $f$).
Let then:

\begin{equation}\label{eq:hd}  
\Horn^{(e)}(\al_*;\rho;\eta) :
=\{(x,y)\in B(0;\eta)   \mid x= \beta(y)=J^{(e)}(\al)(y)+cy^{e}, |c|\leq \rho\} 
\end{equation}

be the  \textit{horn domain}  associated to $\mathcal{D}^{(e)}(\al_*;\rho)$; it is a compact subset of $\C^2$.

The \textit{total asymptotic Gaussian curvature over} $\mathcal{D}^{(e)}(\al_*;\rho)$ is then by definition:
\begin{equation}\label{eq:measure} 
\mer_f(\mathcal{D}^{(e)}(\al_*;\rho)) :=\lim_{\eta\rightarrow 0}\left[ \lim_{\lambda\rightarrow 0}\int_{\{f=\lambda\}\cap \Horn^{(e)}(\al_*; \rho;\eta)}KdS\right] ,
\end{equation}
where $S$ is the surface area and $K$ is the Gaussian curvature. 

The total asymptotic Gaussian curvature over $\mathcal{L}^{(e)}(\al_*)$ as in \eqref{IL} is by definition:
\begin{equation}\label{eq:measure2}
\mer_f(\mathcal{L}^{(e)}(\al_*))\!:=\lim_{\rho\ra \infty}\mer_f(\mathcal{D}^{(e)}(\al_*;\rho)).
\end{equation}

The above definitions are easily extended to the case $e=\infty$, so that
\[
\mer_f(\mathcal{D}^{(\infty)}(\al_*))=\mer_f(\mathcal{L}^{(\infty)}(\al_*))=\mer_f(\{\al_*\})=0.
\]
Then, for the gradient canyon $\GC(\gamma_*) :=\mathcal{L}^{(d)}(\gamma_*)$ one has:

\begin{theorem} \cite[Theorem C]{KKP}
Let $\ga_*$ be a polar, $1< d_{\gr}(\ga_*)\leq\infty$. Then
\be\label{ttl}\mer_f(\GC(\ga_*))=\begin{cases}2\pi[\mu_f(\GC(\ga_*))+\mult(\GC(\ga_*)],&\quad 1<d_{\gr}(\ga_*)<\infty,\\0,&\quad d_{\gr}(\ga_*)=\infty.\end{cases}
\ee \fin
\end{theorem}
%

\section{The arc valleys}\label{NParc} 

\subsection{Arc valleys and gradient canyons}\

We consider  a  Puiseux arc   $\alpha : (\bC, 0) \to (\bC^{2}, 0)$ 
of order $\ord_y(\alpha)\geq 1$. We may assume, modulo transposition and rescaling,  that $\alpha(y)=(\bar\alpha(y),y)$ where $\bar\alpha(y) $ is a fractional power series of order 
$\bar\ord_y(\alpha) \geq 1$;  in the following we identify  $\bar\alpha$ with $\alpha$.

We define the contact degrees: 
$$d_{\con}(\alpha) :=\inf \{d \mid  ||\Grad f(\alpha+ cy^{d}+\hot)||\sim ||\Grad f(\alpha)||,  \mbox{ for  generic }  c\in \bC^2 \}$$ 
and:  
$$ \tilde d_{\con}(\alpha) :=\inf \{d\mid   ||\Grad f(\alpha+cy^{d})||\sim ||\Grad f(\alpha)||,  \, \text{ for generic } \, c\in \bC^2 \}.$$

It follows from the definitions that $\tilde d_{\con}(\alpha)\leq  d_{\con}(\alpha)$.  Let us show that we actually have equality.  
 
If $\|\Grad f(\alpha(y)+u y^e)\|^2 :=D_{(\alpha,e)}(u)y^{2L_{gr}(\alpha,e)}+\hot,\, D_{(\alpha,e)}(u)\not \equiv 0,$
so for generic $u\in \bC^2$ it has the same order  $L_{gr}(\alpha,e)$ which is increasing in $e,$ i.e. 
$\|\Grad f(\alpha(y)+u y^e)\|\sim \|\Grad f(\alpha(y)+vy^r)\|, \forall r>e, v\in \bC^2 \,.$ 

If we set $\beta (y)=\alpha (y) +uy^e$ then $\|\Grad f(\alpha)\| \sim \|\Grad f(\alpha(y)+u y^e)\|=\|\Grad f(\beta)\| \sim \|\Grad f(\beta(y)+v y^r)\|=\|\Grad f(\alpha+uy^d+vy^r)\|$ and so on i.e. the two definitions give the same number (the order of the gradient stabilises, for instance if $r\geq \ord_y(\|\Grad f(\alpha(y))\|^2$).

\m

 Let $\alpha_*\in \bC_{\enr}$ be given, $f(\alpha(y),y)\not \equiv 0$,
 $\alpha$ is not a common Newton-Puiseux root of $f_x$ and $f_y$. Hence, if $q$ is sufficiently large, then
\begin{equation}\label{GGa}
\ord_y(\|\Grad f(\alpha(y),y)\|)=\ord_y(\|\Grad f(\alpha(y)+uy^q,y)\|), \; \forall\, u\in \bC.
\end{equation}

Let $d_{\alpha}$ denote the \textit{smallest number} $q$ such that (\ref{GGa}) holds for \textit{generic} $u\in \bC$. This definition gives the same degree as the previous definition and from now on we will use the later notation.

In case $f(\alpha(y),y)\equiv 0$, we  set  $d_{\alpha} :=\infty$ if $\alpha$ is a multiple root of $f$. 
\begin{definition} 
The \textit{ valley } of $\alpha_*$ is, by definition:
\[ \V(\alpha_*) :=\mathcal{L}^{(d_{\alpha})}(\alpha_*).
\]

We then call $d_{\alpha}$ the \textit{degree} of $\V(\alpha_*)$, or  \textit{ the valley degree of $\alpha$} (since it does not depend of the 
representative $\alpha$ of $\alpha_*$). 

We say that $\mathcal{V}(\alpha_*)$ is \textit{minimal} if $d_{\alpha_*}<\infty$ and if for every arc $\beta$ with $d_{\beta}<\infty$, we have:
\[ \mathcal{V}(\beta_{*}) \subseteq\mathcal{V}(\alpha_*) \Longrightarrow \mathcal{V}(\beta_{*})= \mathcal{V}(\alpha_*).
\]
\end{definition}

\begin{remark} In the case $\gamma$ is a polar the construction above gives  the notion of gradient canyon and canyon degree ($\V(\gamma_{*})=\mathcal{GC}(\gamma_{*}), d_{\gr}(\gamma)=d_{\gamma}$) as introduced in \cite{KKP} and mentioned earlier. 
If $\alpha$ is in a canyon then its valley coincides with the canyon.


 \end{remark}


%


\subsection{Newton polygon}
Let $\alpha$ be a given arc, $d_{\alpha}<\infty$.

We can apply a unitary transformation, if necessary, so that $\alpha\in \F_{\ge 1}$, $T(\alpha_*)=[0:1]$.

We then change variables (formally):
\begin{equation}\label{3b}
Z :=z-\alpha(w),\quad W :=w,\ \ 
F(Z,W) :=f(Z+\alpha(W),W).
\end{equation}

Since $\alpha\in \F_{\ge 1}$, i.e., $\ord(\alpha)>1$, it is easy to see that
\begin{equation}\label{D}
\|\Grad_{z,w} f\|\sim\|\Grad_{Z,W} F\|,\ \
\Delta_f(z,w)=\Delta_F(Z,W)+\alpha''(W)F_X^3.
\end{equation} 
 \m

The Newton polygon $\PP(F)$ is defined in the usual way, as follows. Let us write
$$F(Z,W)=\sum c_{iq}Z^iW^q,\quad c_{iq}\ne 0,\quad (i,q)\in \Z\times\Q.$$
A monomial term with $c_{iq}\ne 0$ is represented by a ``\textit{Newton dot}" at $(i,q)$. We shall simply call it a \textit{dot} of $F$. The \textit{boundary} of the convex hull generated by $\{(i+u,q+v)|u,\,v\geq 0\}$, for all dots $(i,q)$, is the Newton Polygon $\PP(F)$, having edges $E_i$ and angles $\theta_i$, as shown in Fig.\ref{fig:npF}. In particular, $E_0$ is the half-line $[m,\infty)$ on the $Z$-axis. \footnote{In\cite{kuo-lu}, \cite{kuo-par}, this is called the Newton Polygon of $f$ \textit{relative to} $\alpha$, denoted by $\PP(f,\alpha)$.}

For a line in $\R^2$ joining $(u,0)$ and $(0,v)$, let us call $v/u$ its \textit{co-slope}. Thus
$$\textit{co-slope of}\;\;E_s=\tan\theta_s.$$

Some elementary, but useful, facts are:
\begin{itemize}
\item If $i\geq 1$, then $(i,q)$ is a dot of $F$ \textit{if and only if}
$(i-1,q)$ is one of $F_Z$.
\item When  $f(\alpha(w),w)\not\equiv 0$, we know $F(0,W)\not\equiv 0$. Let us write
\begin{equation} \label{3k}F(0,W)=aW^h+\hot, \quad a\ne 0, \;\,
h :=\ord_W(F(0,W))\in \Q.\end{equation}
Then $(0,h)$ is a vertex of $\PP(F)$, $(0,h-1)$ is one of $\PP(F_W)$. (See Fig.\,\ref{fig:npZ}.)
\end{itemize}

\begin{figure}[htt]
 \begin{minipage}[b]{.4\linewidth}
   \centering
   \begin{tikzpicture}[scale=0.44]
   
   \draw(0,0) -- (0,11.5);
   \path (0,11.5) node[above] {\tiny{W}};
  
   \path (7.5,0) node[below] {\tiny{$m$}};   
   \draw(0,0) -- (8.1,0);
    \path (8,0) node[right] {\tiny{Z}};
     
  \draw[dotted] (0.5,8)--(2,5);
    
   \draw plot[mark=*] coordinates {(2,5) (2.7,4)};
   \draw plot[mark=*] coordinates {(0.5,8) (0,11)};
    
   \draw plot[mark=*] coordinates {(5.6,1.4) (7.6,0)};
    
    \draw[dotted] (3.6,3) -- (4.8,2);

 \draw[dotted] (1.5,4) -- (2.8,4);
  \path (2,3.7) node[above] {\tiny{$\theta_s$}};
 
   \path (6.3,-0.3) node[above] {\tiny{$\theta_1$}};
    
    \path (7.1,0.35) node[above] {\tiny{$E_1$}};
    
    \path (3.7,4.8) node[left] {\tiny{$E_s$}};
    
     \path (0.6,8) node[right] {\tiny{$(m_{\top},q_{\top})$}};
    
    \path (2.2,9.2) node[left] {\tiny{$E_{\top}$}};
    
    \draw (6.8,0) arc (180:150:1);
    
    \end{tikzpicture}
   \caption{$\PP(F)$}\label{fig:npF}
    \end{minipage}
    \begin{minipage}[b]{.4\linewidth}
   \centering
 \begin{tikzpicture}[scale=0.4]
  
   \draw(0,0) -- (0,13);
   \path (0,13) node[above] {\tiny{$W$}};
    
    \draw[dotted](7,0.5) -- (10.5,0.5);
    
    \path (8.8,0.45) node[above] {$\theta_{\top}$};
    
    \draw(0,0) -- (12,0);
    \path (12,0) node[below] {\tiny{$Z$}};
     
    \draw[dotted] (5,0) -- (5,11);
    
    \draw plot[mark=*] coordinates {(0,12) (9,2) (10.5,0.5)};
    
    \draw plot[mark=*] coordinates {(9,2) (7,5) (5.5,8)};
    
    \draw[dotted] (5,11) -- (9,2);
    
    \draw[dashed] (5.5,8) -- (5,11);
    
    \draw[dashed] (0.5,8) -- (0,11);
    
    \path (5.1,9.5) node[right] {\tiny{$L$}};
    
    \path (0.1,9) node[right] {\tiny{$L^*$}};
    
    \path (4.4,11) node[right] {$\circ$};
    
    \path (0,11) node {$\circ$};
    
    \path (0,12) node[right] {\tiny{$(0,h)$}};
    
    \path (0,10.9) node[left] {\tiny{$(0,h_{\alpha})$}};
    
    \path (9,2) node[right] {\tiny{${(\widehat{m}_{\top},\widehat{q}_{\top})}$}};
    
    \path (10.5,0.5) node[right] {\tiny{$(m_{\top},q_{\top})$}};
    
    \path (5.5,8) node[right] {\tiny{$(m^*+1,q^*)$}};
    
    \path (5.3,11) node[right] {\tiny{$(1,h_{\alpha})$}};
    
    \path (5,0) node[below] {$1$};
    
    \path (2.3,7) node[right] {\tiny{$E_{\top}$}};
    
    \draw[dotted] (9.5,0.5) arc (180:145:1.3); 
     
    \end{tikzpicture}   
     
    \caption{$\PP(F)$~vs~$\PP(F_Z)$.} \label{fig:npZ}
    \end{minipage}
    \end{figure}

\begin{notations}\label{top}
In the case $d_{\alpha}<\infty$ ,  i.e. either $f(\alpha (w),w) \not\equiv 0$ or $f(\alpha (w),w)=0, f_z(\alpha(y),y)\not\equiv 0,$ let $E_{\top}$ denote the edge whose left vertex is, in the first case $(0,h)$, $h$ as in (\ref{3k}) or $(1,h')$ in the second case,  and right vertex is $(m_{\top}, q_{\top})$, as shown in the figures. We call it the \textit{top} edge; the angle is $\theta_{\top}$.  In the case $f(\alpha(w),w)\equiv 0$ the top edge $E_{\top}$ is not ending on $z=0$ but on $z=1$,  except   $\alpha$  is  a multiple root of $f$ in which case we  precisely have $d_{\alpha}=\infty$. However, except in the latter case, we always extend informally the top edge to virtually cut $z=0$ at $(0,h)$.

\end{notations}
 
Let $(\widehat{m}_{\top},\widehat{q}_{\top})\ne (0,h)$ be the dot of $F$ on $E_{\top}$ which is
\textit{closest} to the left end of    $E_{\top}$  ( which is $(0,h)$ if $\alpha$ not a multiple root of $f$) . (Of course, $(\widehat{m}_{\top},\widehat{q}_{\top})$ may coincide with $(m_{\top}, q_{\top})$.) Then, clearly,
\[
\label{3d}2\leq \widehat{m}_{\top}\leq m_{\top},\quad \frac{h-\widehat{q}_{\top}}{\widehat{m}_{\top}}=\frac{h-q_{\top}}{m_{\top}}=\tan\theta_{\top}.\]     

Now we draw a line $L$ through $(1,h_\alpha), h_\alpha \leq h-1,$ with the following two properties (in particular this  defines $h_\alpha$):
\begin{itemize}
	\item[(a)] If $(m',q')$ is a dot of $F_Z$, then $(m'+1,q')$ lies on or above $L$;
	\item[(b)] There exists a dot $(m^*,q^*)$ of $F_Z$ such that $(m^*+1,q^*)\in L$. (Of course, $(m^*+1,q^*)$ may coincide with $(\widehat{m}_{\top},\widehat{q}_{\top})$.)
	\item[(c)] $h_\alpha$ is the largest with these properties.
\end{itemize}

\begin{lemma}\label{LH}
Let $\sigma^*$ denote the co-slope of $L$. Then
\begin{equation}\label{tan} (i)\;d_{\alpha}=\sigma^*;\,(ii)\,\;\sigma^*\geq \tan\theta_{\top};\;(iii)\;  \sigma^*=\tan\theta_{\top}\Leftrightarrow (1,h_\alpha)\in E_{\top}.
\end{equation}

All dots of $F_W$ lie on or above $L^*$, $L^*$ being the line through $(0,h_\alpha)$ parallel to $L$.

 In the case $\tan\theta_{\top}>1$, $(0,h-1)$ may be  the only dot of $F_W$ on $L^*$ (exactly when $h_\alpha=h-1$ and $\alpha$ not a root of $f$).
\end{lemma}


\begin{notations}\label{no}
Take $e\geq 1$. Let $\oa(e)$ denote the weight system: $\oa(Z)=e$, $\oa(W)=1$.

Let $G(Z,W^{1/N})\in \bC\{Z,W^{1/N}\}$ be given. Consider its weighted Taylor expansion relative to this weight. We shall denote the \textit{weighted initial form} by $\I_{\oa(e)}(G)(Z,W)$, or simply $\I_\oa(G)$ when there is no confusion.

If $\I_\oa(G)=\sum a_{ij}Z^iW^{j/N}$, the \textit{weighted order} of $G$ is $\ord_\oa(G) :=ie+\frac{j}{N}$.
\end{notations}

\begin{proof}Note that ($ii$) and ($iii$) are clearly true, since $(1,h_\alpha)$ lies on or above $E_{\top}$.

Next, if $(i,q)$ is a dot of $F_W$, then $(i,q+1)$ is one of $F$, lying on or above $E_{\top}$. Hence, by ($ii$), all dots of $F_W$ lie on or above $L^*$.

It also follows that if $\tan\theta_{\top}>1$, then $(0,h-1)$ may be  the only dot of $F_W$ on $L^*$. 

\m

Let us show $(i)$. 

It is easy to see that if $\tan\theta_{\top}=1$, then $d_{\alpha}=1$.

It remains to consider the case $\sigma^*>1$. By construction, 
\begin{equation}\label{LI}\ord_y(\|\Grad f(\alpha(y),y)\|)=h_\alpha.
\end{equation}

Let us first take weight $\oa :=\oa(e)$ where $e\geq \sigma^*$. In this case, since $\sigma^*>1$:
\[\I_{\oa}(F_W)(Z,W)=ahW^{h-1},\  \ord_W(F_Z)= h_\alpha, \mbox{ if }  h_\alpha \leq h-1, \mbox{ and } > h-1 \mbox{ otherwise,}\] 
where $a$, $h$ are as in (\ref{3k}), $ah\ne 0$. Hence for generic $u\in \bC$,
$$ \ord_W(F_W(uW^e,W))=h-1,\quad \ord_W(F_Z(uW^e, W))= h_\alpha.$$

It follows that $d_{\alpha}\leq \sigma^*$. It remains to show that $\sigma^*>d_{\alpha}$ is impossible.

\m

Let us take $\oa(e)$ with $e<\sigma^*$. Note that $(m^*,q^*)$ is a dot of $F_Z$ on $L^*$, where $(m^*+1,q^*)$ is shown in Figure \ref{fig:npZ}.
Hence, for generic $u$,
$$\ord_W(F_Z(uW^e, W))<h_\alpha,\quad \ord_W(\|\Grad F(uW^e,W)\|)<h_\alpha.$$

Thus, by (\ref{LI}), we must have $d(\alpha)>e$. This completes the proof of Lemma \ref{LH}.
\end{proof}

\begin{example}\label{ExampleGC}
For $F(Z,W)=Z^4+Z^3W^{27}+Z^2W^{63}-W^{100}$ and
$\gamma=0$, $\mathcal{NP}(F)$ has only two vertices $(4,0)$,
$(0, 100)$, while $\mathcal{NP}(F_Z)$ has three: $(3,0)$,
$(2,27)$, $(1,63)$. The latter two and $(0,99)$ are collinear, spanning $L^*$; $h=100$, $\sigma^*=(99-27)/2=36$.
\end{example}

In the following when we say ``arc'' we mean a complex Puiseux arc.



\begin{proposition}\label{p:alpha}
\begin{enumerate}
\rm \item \it
For any arc $\alpha$, there is some polar $\gamma$ of $f$ such that $d_{\gamma}\ge d_{\alpha}$.
\rm \item \it
For all $\beta \in \mathcal{GC}(\gamma_{*})$   one has $d_{\beta}=d_{\gr}(\gamma)$. 
\rm \item \it
For any $\al \in \GC (\gamma_{*})$ one has $f(\al(y))=ay^h+\hot$, where $a$ and $h$ depend only on the canyon.
\end{enumerate}
\end{proposition}
\begin{proof}   
(a). By using the Newton polygon relative to $\alpha$,  $\PP(f,\alpha)$,  see Figure \ref{fig:npF},  we observe that a polar can be obtained by pushing forward along $L^*$. Namely  we construct a root of $f_x$ starting from  $\PP(f_x,\alpha)$) by the Newton-Puiseux algorithm. 
This procedure adds up terms of degree  at least $d_{\alpha}$, so we end up with at least one polar of the form $\gamma=\alpha +cy^{d_{\alpha}}+ \hot$, where $c$ is a root of the associated polynomial in $x$ (i.e. the derivative of the de-homogenisation of the polynomial associated to $L^*$). 
We then get $\ord_y(\alpha(y)-\gamma(y))\geq d_\alpha$, hence  $d_{\alpha} \leq d_{\gr}(\gamma)$ for any such polar.

Hence, starting with $\alpha$  one constructs polars by the diagram method and the process is not necessarily unique.  Nevertheless, all such polars are clearly in the valley of $\alpha$. 


\noindent
(b).  Our assumption implies that $\beta=\gamma +cy^{d_{\gr}(\gamma)}+ \hot$  for some $c\in \bC$, hence the Newton polygons 
$\PP(f,\beta)$ and $\PP(f,\gamma)$
will have the corresponding $L$ parallel, see Figure \ref{fig:npF},  and thus the same co-slope, which is $d_{\gr}(\gamma)$. 

\noindent
(c). We have by definition $f(\gamma(y), y) = ay^{h}+\hot$ and by our assumption $\al(y) = \gamma(y) + cy^{d} + \hot$. Thus
\begin{equation}\label{eq:h-coefficient}
 f(\gamma(y) + cy^{d} + \hot, y) = ay^{h}+ \cdots + \alpha(c)y^{d+h-1}+\hot ,
 \end{equation}
where the first terms depend only on the canyon (and not on the perturbation of $\gamma$), in particular the dependence of $c$ starts at the degree $d+h-1$.

\end{proof}
\begin{remark} \label{r:generic}
Point (b) above holds for the gradient canyons but it is not necessarily true for arbitrary valleys. More precisely, in case of a valley $\V(\gamma_{*})$, the claim (b) holds only for arcs $\beta=\gamma +cy^{d_\gamma}+ \hot$  where  the coefficient $c\in \bC$ is generic.
\end{remark}
\begin{remark}
In general, given  $\al\in \GC (\gamma_{*})$ with $\ord_y(\al)=1$, to put it in the form $(\tilde\al(y),y)$ requires a rescaling of $y$ (i.e. replacing $y$ by $cy$ for some $c\neq 0$) and this yields $f(\al(y))=ac^hy^h+\hot$
\end{remark}

\begin{corollary}\label{c:alpha} 
The function  $\alpha \mapsto d_{\alpha}$  has its local maxima at the polars $\gamma$ of $f$ with $d_{\gamma} >1$. 
\fin
\end{corollary}

\subsection{Analytic invariants}

Let us show what are the canyon type invariants
up to analytic equivalence, before entering the more involved study of the bi-Lipschitz invariants in \S \ref{s:disks}.
So let $f= g \circ \varphi$, for some bi-holomorphic map $\varphi : (\bC^{2}, 0) \to (\bC^{2}, 0)$.  

For some arc $(\alpha(y), y)$, we then have $\varphi(\alpha(y), y) = (\varphi_{1}(\alpha(y), y),
 \varphi_{2}(\alpha(y), y))$ with \\
 $\ord_{y}\varphi_{2}(\alpha(y), y) =1$, hence we may write   $\varphi(\alpha(y), y) = (\beta(\bar y), \bar y)$ for some arc $\beta$,  where $\bar y := \varphi_{2}(\alpha(y), y)$. 
Then we have:
\begin{theorem}\label{t:degree}
For any  polar $\gamma_{f}$ of $f$ there exists a polar $\gamma_{g}$ of $g$ such that  
\[\varphi(\GC(\gamma_{f*})) = \GC (\gamma_{g*}) \]
and the canyon degrees are the same.
\end{theorem}
\begin{proof}
Let us prove first:
\begin{lemma} \label{l:d}\label{p:varphi1} 
 $d_{\varphi(\alpha)}= d_{\alpha}$ and  $\varphi(\mcv(\alpha_{*})) = \mcv (\varphi(\alpha(y), y)_{*})$.
 \end{lemma}
 \begin{proof} 
We have $\varphi (\alpha + cy^{d_{\alpha}}, y)) = \varphi (\alpha) + a(c)\bar y^{d_{\alpha}} +  \hot$
and the following equivalence:
 \[  \grad f(\alpha + cy^{d_{\alpha}}, y) \simeq_{ord} \grad g(\varphi (\alpha + cy^{d_{\alpha}}, y)) 
 \]
 \[ = \grad g(\varphi (\alpha) + a\bar y^{d_{\alpha}} +  \hot,  \bar y).
 \]
By the definition \eqref{GGa} of the degree, one may consider some generic coefficient $c\in \bC$, and its transform $a(c)$ which is also generic. Then by Remark \ref{r:generic} we 
may apply Proposition \ref{p:alpha}(b) for valleys,  and get the inequality 
$d_{\alpha} \ge d_{\varphi(\alpha)}$.   

We apply the same to  $\varphi^{-1}$ instead of  $\varphi$ and obtain the converse inequality $d_{\varphi(\alpha)}\ge d_{\varphi^{-1}(\varphi(\alpha)) }= d_{\alpha}$, thus our first claim is proved.

Next, we have: 
\[ \ord_{y}\| \varphi ((\alpha(y), y) - \varphi ((\beta(y), y)\| = \ord_{y}\| (\alpha(y), y) - (\beta(y), y)\| .
\]
By using the just proven equality of degrees we  get the second claimed equality.
\end{proof}

We have   $\varphi(\GC(\gamma)) = \mcv (\varphi(\gamma(y), y))$ by Lemma \ref{p:varphi1}.  After Proposition \ref{p:alpha} we may associate to $\varphi(\gamma(y), y)$ some polar $\gamma_{g}$  in the valley of $\varphi(\gamma(y), y)$ with $d_{\gamma_{g}}\ge d_{\varphi(\gamma)}$, and therefore $\mcv (\varphi(\gamma(y), y)) \supset  \mcv (\gamma_{g}) = \GC (\gamma_{g})$. 

We apply $\varphi^{-1}$ and get similarly: $\GC(\gamma_{*})\supset \varphi^{-1}(\GC(\gamma_{g*})) \supset  \GC(\gamma_{f*})$ for some $\gamma_{f}$ constructed  like in Proposition \ref{p:alpha}, with $d_{\gamma_{f}} \ge d_{\varphi^{-1}(\gamma_{g})}$.
According  to the minimality principle of polar canyons Theorem \ref{thmB}, we must have equality:  $\GC(\gamma_{*})=   \GC(\gamma_{f*})$
and $d_{\gamma_{f}} = d_{\gamma}$. Consequently we get that $\varphi(\GC(\gamma_{*})) =  \GC (\gamma_{g*})$ and the degrees are equal.
 \end{proof}
 
 While analytic maps do not preserve polars, we may now prove the analytic invariance of the canyons:
\begin{theorem}\label{t:main1}
If $f = g\circ \varphi$ with $\varphi$ bi-holomorphic, then $\varphi$ transforms canyons into canyons by preserving their degrees and multiplicities.
\end{theorem}

\begin{proof}

Theorem \ref{t:degree} shows that $\varphi $ sends a gradient canyon to a gradient canyon by preserving the degree.
The preservation of the multiplicity follows from Proposition \ref{p:alpha}(c).

 \end{proof}
 
 


\section{Generic polars}\label{s:genpolars}


Based on Langevin's approach \cite{La-cmh} to the integral of the curvature of the Milnor fibre of a function of  $n$ complex variables, Garcia Barroso and Teissier \cite{GT} gave a method to detect the concentration of curvature on the Milnor fibre of a function germ in 2 variables.  Using Langevin's exchange formula which interprets the curvature in terms of polar curves, they showed that the intersections of the Milnor fibre with all generic polar curves  is concentrated in certain small balls, and hence the curvature too. 
 
 More recently,
Koike, Kuo and P\u aunescu \cite{KKP} adopted  a new viewpoint by looking into the curvature formula itself and
studying its variation over the space of arcs.  Their method uses the gradient canyons  and provides sharper localization of the ``A'Campo bumps''  i.e. maxima of curvature.

 We shall find here the relations between the results obtained in \cite{GT} and  in \cite{KKP}.
 Let $\gamma_{0}$ denote a solution of the  equation $f_{x}(\gamma_{0}(y), y)=0$.
 Let $l_{\tau} \subset \bC^{2}$ of coordinates $(x,y)$ denote the line $\{y-\tau x=0\}$, and call it \emph{the line of co-direction $\tau$}.
The polars $\gamma_{\tau}$ are the solutions of the equation:
\begin{equation}\label{eq:tau}
 (f_{x}+\tau f_{y})(\gamma_{\tau}(y), y)=0. 
 \end{equation}

\begin{theorem}\label{t:tau}
The gradient canyon $\GC(\gamma_{\tau *})$ does not depend on the direction $\tau \in \bC$, i.e. $\GC(\gamma_{0*}) =\GC(\gamma_{\tau *})$, $\forall \tau \in \bC$.
The canyon degree $d_{\gamma_{0}}$  is the lowest exponent from which the polar expansions $\gamma_{\tau}$ start to depend of $\tau$.
The multiplicities $m_{\gamma_{\tau}}$ do not depend of $\tau$. 
\end{theorem}

\begin{proof} 
In case $d_{\gamma_{0}}>1$ we consider the function $(f_{x} +\tau f_{y})(x + \gamma_{0}(y), y)$ from which we want to construct
a solution of \eqref{eq:tau} by the method of ``pushing forward'' in the Newton diagram, as explained in the proof of Proposition \ref{p:alpha}.

The polars associated to the direction $\tau$  are the Newton-Puiseux zeroes of the function
$g(x,y) = (f_{x} +\tau f_{y})(x + \gamma_{0}(y), y)$ translated by $\gamma_{0}$, namely $\gamma_{\tau}(y) := x(y)+ \gamma_{0}(y)$.
The top edge $E_{\top}$ of the Newton polygon of $g$ is parallel to the segment $L$ defined as in Figure \ref{fig:npZ}, taking $\alpha = \gamma_{0}$.
 Whenever $d_{\gamma_{0}}>1$, the segment $L$ has only one dot which depends on $\tau\not=0$, namely the one corresponding to the 
 monomial $a(\tau)y^{h-1}$ (which comes from the contribution $\tau f_{y}$). Therefore the edging forward process will start with the initial term of the form $c(\tau)y^{d_{\gamma_{0}}}$ in order to annihilate $a(\tau)y^{h-1}$. Thus the Newton-Puiseux zero of $g$ will be of the form
 $x(y) = c(\tau)y^{d_{\gamma_{0}}} +\hot$, hence $\gamma_{\tau}(y) = \gamma_{0}(y) + c(\tau)y^{d_{\tau_{0}}} +\hot$ is a polar associated to $\tau$. This shows in particular that the constructed solution $\gamma_{\tau}$ is in the same gradient canyon as the  polar $\gamma_{0}$.

Note that the generic polars that we have  constructed $\gamma_{\tau}(y) = \gamma_{0}(y) + c(\tau)y^{d_{\tau_{0}}} +\hot$ 
are in the canyon of $\gamma_0$ and therefore $f(\gamma_{\tau},y)=ay^h+\hot$, thus the initial term is constant in the canyon, and in particular the exponent $h$ does not depend of $\tau$.


By construction the number of  roots $x=x(y)$ of $g(x,y)=0$ is $m_{\gamma_{0}}$, where 
 $(m_{\gamma_{0}} + 1, r)$ is the initial (lowest) dot of $L$.  Consequently  $(m_{\gamma_{0}},r)$ is the initial dot of $E_{\top}$, 
 hence 
 \[ \frac{h-1-r}{m_{\gamma_{0}}} = d_{\gamma_{0}}
 \]
 and $m_{\gamma_{0}}$ is the multiplicity of the canyon, i.e. the total number of  polars in the canyon  $\GC(\gamma_{0})$, for any $\tau \in \bC$.

\end{proof}


\subsection{Garcia-Barroso and Teissier's approach \cite{GT}}\label{ss:GT}

Let us  recall some of the results obtained in \cite{GT} by following their original notations.
 
 \m
\noindent
(1). Let $P_{q}(\tau) = (x(t), y(t))$ where $x(t) = t^{m}$, $y(t)= at^{m} + \hot$, be a minimal parametrisation of  an irreducible branch of the polar curve with respect to a direction $\tau \in \bP^{1}$. Here  $m_{q}= m_{q}(\tau)$ is  the multiplicity at 0 of $P_{q}(\tau)$. Teissier had proved that the family $P_{q}(\tau)$ depending of $\tau$  is equisingular for generic $\tau$, thus the multiplicity $m_{q}(\tau)$ is constant for generic $\tau$.  In the following we consider $\tau$ in such a generic set.
 
 Barroso and Teissier show in \cite{GT}
that the polars fall into subsets called ``packets''  indexed by the black vertices of the Eggers diagram of $f$,
such that they have the same contact with all branches of the curve $C:= \{ f=0\}$.    
Such a ``packet'' of polars is the set of polars from a certain union of canyons.
\m

\noindent
(2).  By \cite[Theorem 2.1]{GT}, the coefficients of $P_{q}(\tau)$ depend on $\tau$
 only from a certain well-defined exponent of $t$.
Let  $g_{q}$ denote the first exponent of $y(t)$ the coefficient of which depends of $\tau$. It is shown that all the polars in the same packet have the same  exponent $g_{q}$ and this is denoted by $\gamma_{q}$, cf \cite[pag. 406]{GT}.  

\m
\noindent
(3). Moreover,  in the development of $f(t^{m_{q}}, y_{q}(t, \tau))$, the first exponent the coefficient of which depends of $\tau$ 
is $e_{q}+ g_{q}$.  The geometric significance of $e_{q}$ is given by the identity
 \[ \mult_{0}(C, P_{q}(\tau))=  e_{q} + m_{q},\]
 where  $e_{q}= \mu_{q}(f)$ is a partial Milnor number  in the sense that,   by Teissier's formula for the polar multiplicity,  the sum of $e_{q}$'s over all polar in the packet and over all packets is equal to the Milnor number $\mu_{f}$.
 
 \m
 \noindent
(4). The concentration of points of intersection $P_{q}(\tau) \cap \{ f= \lambda\}$ on the Milnor fibre, for all generic $\tau$ 
 and as $\lambda$ approaches 0,
is equivalent  to the concentration of curvature, according to Langevin's approach \cite{La-cmh}.
 In order to locate  the zones of concentration on the Milnor fibre, i.e. the centers of the balls and their radii, Barroso and Teissier invert the convergent series $\lambda = \lambda(t)$ and expresses the coordinates $(x(t),y(t))$ as functions of $\lambda$ (see \cite[(5), page 408]{GT}).

Let us now see what are the relations between these invariants and those defined in \cite{KKP} and in our previous sections.

\subsection{A  dictionary}\

We have  shown that the value $h= \ord_{y}f(\gamma(y), y)$  is  the same for all polars in some canyon (Proposition \ref{p:alpha}). Therefore
$m_{\gamma}h := \ord_{y}f(\gamma(y^{m_{\gamma}}), y^{m_{\gamma}}) = \mult_{0}(C, \gamma_{*})$, where $m_{\gamma}$ is the multiplicity of the polar considered with its multiple structure. Note the difference to  \cite{KKP} and the preceding sections,  where by ``polar'' we mean with multiplicity 1, and precisely $m_{\gamma}$ such polars have the same image $\gamma_{*}$.

On the other hand, by \S \ref{ss:GT}(3),  from the \cite{GT} viewpoint we have $\mult_{0}(C, \gamma) = e_{\gamma} + m_{\gamma}$.
 We therefore conclude:
 \begin{equation} \label{eq:mu}
 e_{\gamma} = m_{\gamma}(h-1)
 \end{equation}
 which can be identified with a partial sum of \eqref{M}.
 
 By Theorem \ref{t:tau}:
 \[
 \gamma_{\tau}(y) = \gamma_{0}(y)+ c(\tau) y^{d_{\gamma_{0}}} + \hot \]
 which implies
 \[ f(\gamma_{\tau}(y), y) = a y^{h} + \cdots + u(\tau)y^{d_{\gamma_{0}} +h-1} + \hot
 \]
where $u(\tau)$ is the first coefficient which depends of $\tau$; thus:
 \[ f(\gamma_{\tau}(y^{m}), y^{m}) = a y^{mh} + \cdots + u(\tau)y^{md_{\gamma_{0}} +m(h-1)} + \hot
 \]

By \cite[Lemma 2.2]{GT}:
\[ f(\gamma_{\tau}(y^{m}), y^{m}) = a y^{mh} + \cdots + u(\tau)y^{e_{q} + g_{q}} + \hot
 \]
using the notations $e_{q}$ and $ g_{q}$ from  \S \ref{ss:GT}(3).  

We obtain:

\[e_{q} + g_{q} =m_{q}d_{\gamma} +m_{q}(h-1)\]
hence
\[ g_{q}= m_{q} d_{\gamma}
\]
which shows that the exponent $g_{q}$ of \cite{GT} reminded at  \S \ref{ss:GT}(2) is essentially the same as the degree $d_{\gamma}$ of the canyon, i.e.  modulo multiplication by the multiplicity $m_{\gamma}$.


\section{The correspondence of canyon disks}\label{s:disks}

We consider in this section a gradient canyon $\GC(\ga_*)$ of degree $d_{\gamma} >1$.
Let $D^{(e')}_{\gamma_{*}, \e}(\lambda;\eta)$ be the union of disks in the Milnor fibre $\{ f=\lambda\}\cap B(0;\eta)$ of $f$ defined as follows (see  \eqref{eq:hd} for the definition of the $\Horn$):
\[ D^{(e')}_{\gamma_{*}, \e}(\lambda;\eta) := \{ f=\lambda\} \cap \Horn^{(e')}(\gamma_*;\e;\eta),
\]
for some rational $e'$  close enough to $d_{\gamma}$, with $1< e' <d_{\gamma}$,  for some small enough $\e>0$, and  where by ``disk'' we mean ``homeomorphic to an open disk''.  In addition, we ask that $d< e' <d_{\gamma}$ for any other canyon degree $d < d_{\gamma}$.

We have:
\begin{equation}\label{eq:horn}
 \bigcap_{e' \in \bQ, e' \to d_{\gamma}}D^{(e')}_{\gamma_{*}, \e}(\lambda;\eta) = \{ f=\lambda\} \cap \Horn^{(d_{\gamma})}(\gamma_*;\e;\eta)
\end{equation}
and we shall write $D_{\gamma_{*}}(\lambda)$  in the following as shorthand for $D^{(e')}_{\gamma_{*}, \e}(\lambda;\eta)$, keeping in mind the parameters $e', \e, \eta$.

\smallskip

By \cite[Lemma 6.6]{KKP} 
we have--see also \eqref{eq:measure} and \eqref{eq:measure2}:
\begin{equation}\label{eq:curvintegralcanyon}  
 \mer_f(\cD^{(e')}(\gamma_{*}, \e)) = \mer_f(\cL^{(e')}(\gamma_*)) = \mer_f(\GC(\ga_*)),
\end{equation}
and moreover $\cL^{(e')}(\gamma_*)$ does not contain any other polar canyon besides  $\GC(\ga_*)$. 

This result means that a certain  part of the curvature of the Milnor fibre is concentrated in the union  $D_{\gamma_{*}}(\lambda) = \cup_{i}D_{\gamma_{*, i}}(\lambda)$, of connected disks $D_{\gamma_{*, i}}(\lambda)$.  The number of disks is the intersection number  $\mult_{0} (\{ f=0\}, \overline{\gamma_{*}})$, where $\overline{\gamma_{*}}$ is the truncation of the polar at the order $d_{\gamma}$.
Here we have to understand $\gamma_{*}$ as image of $\gamma_{*}$, which is thus the same image for all conjugates of $\gamma$, and similarly for the truncations.  There might be non-conjugate polars in the same canyon, and then the (centers of) the disks are the same.

 The centers and the radii of the disks $D_{\gamma_{*, i}}(\lambda)$ are  given more explicitly in \cite[\S 3.1]{GT}, as we shall briefly describe in the following. 
 
 First, one has to express the coordinates $x=\gamma(y)$  and $y$  in terms of $\lambda$.  One obtains 
 an expansion $(x(\lambda), y(\lambda))$ with complex coefficients:
 \begin{equation}\label{eq:centers}
    \left(  \sum_{i=m}^{\infty}  \alpha_{i}\lambda^{\frac{i}{mh}},  \sum_{i=m}^{\infty}  \beta_{i}\lambda^{\frac{i}{mh}} \right).
 \end{equation}
 For polars $\gamma_{\tau}$ depending of the generic direction $\tau$, as we have discussed in \S \ref{s:genpolars}, the first coefficients of \eqref{eq:centers} which depend of  $\tau$ are $\alpha_{md}$ and $\beta_{md}$, where $m$ is the multiplicity of $\gamma_{\tau}$ and $d$ is its polar degree, both of which are independent of the generic $\tau$, by Theorem \ref{t:tau}.  Note that $h$ is also independent of $\tau$.
 
The centers of the disks\footnote{The number of the disks is $mh$.} $D_{\gamma_{*, i}}(\lambda)$ are then the truncations of \eqref{eq:centers} up to the order $(md-1)/mh$. 
They are conjugated by $\lambda\to \omega^{i}\lambda$, where $\omega$ is a primitive root of order $mh$.
The radii of the disks are of the form $r | \lambda|^{d/h} \sim | y |^{d}$, where 
$r\in \bR_{+}$ depends on the compact subset of $\bP^{1}$ in which $\tau$ varies.  The distance between  two (centers of) such disks  is 
of order $| \lambda|^{1/h}\sim | y |$.


\begin{theorem}\label{t:localcurv}
Let $f= g\circ \varphi$, where $\varphi$ is bi-Lipschitz. Then:
\[ \lim_{\lambda\to 0} \frac{1}{2\pi} \int_{\varphi(D_{\gamma_{*, i}}(\lambda))} K_{g} \ dS \ \ge 1\]
 for any $i$.
\end{theorem}

The proof consists of several steps.
\begin{lemma}\label{l:totalcurvconcentr}
$ \displaystyle \lim_{\lambda\to 0} \frac{1}{2\pi} \int_{D_{\gamma_{*, i}}(\lambda)} K_{f} \ dS \ $
 is a positive integer, for any $i$.
\end{lemma}
\begin{proof}
By \eqref{eq:curvintegralcanyon} the canyon $\GC(\ga_*)$ concentrates the total curvature:
\[ \frac{1}{2\pi} \mer_{f}\GC(\ga_*) = \sum \mult_{0}(\{ f=0\}, \gamma'_{*}),
\]
where the sum is taken over all polars $\gamma'_{*}$ in the canyon $\GC(\ga_*)$. This is a multiple of  the number $\mult_{0} (\{ f=0\}, \overline{\gamma_{*}})$ of the disks of the canyon $\GC(\ga_*)$. These disks contain 
all the intersections  of the Milnor fibre with the polars $\gamma_{\tau *}$, for $\tau$ in some dense subset of  a compact $K(\lambda) \subset \bP^{1}$ which tends to $\bP^{1}$ when $\lambda \to 0$. On the other hand, as we have seen just above, these disks are conjugate. Therefore, when $\lambda \to 0$, each such disk concentrates the same total curvature, which must be a positive integer (modulo $2\pi$).
\end{proof}

We need the interpretation of Lemma \ref{l:totalcurvconcentr} in terms of the directions $\tau\in \bP^{1}$. By applying Milnor's exchange formula (see Langevin's paper \cite{La-cmh})
we have the equalities:
\begin{equation}\label{eq:exchange}
   \frac{1}{2\pi} \int_{D_{\gamma_{*, i}}(\lambda)} K_{f} \ dS  = \frac{1}{2\pi} \int_{D_{\gamma_{*, i}}(\lambda)} |\jac \nu_{\bC}|^{2} \ dS  
  \end{equation}   
where $\jac \nu_{\bC}$ denotes the Jacobian determinant of the complex Gauss map. In turn, this is equal, cf \cite{La-cmh}, to:
\begin{equation}\label{eq:degree}
 \frac{1}{2\pi} \int_{ D_{\gamma_{*, i}}(\lambda)} \nu_{\bC}^{*} dp  =
  \frac{1}{2\pi} \int_{\nu_{\bC}( D_{\gamma_{*, i}}(\lambda))} \deg  ({\nu_{\bC}}_{| D_{\gamma_{*, i}}(\lambda)}) \ dp 
  \end{equation} 
where the last equality follows from the  constancy of the degree  $\deg  ({\nu_{\bC}}_{| D_{\gamma_{*, i}}(\lambda)})$ by Theorem \ref{t:tau}.

Since $2\pi$ represents the volume of $\bP^{1}$, we have proved:

\begin{lemma}\label{l:dense}
The image of the disk $D_{\gamma_{*, i}}(\lambda)$
 by the Gauss map $\nu_{\bC}$, as $\lambda$ tends to 0, is a dense subset of $\bP^{1}$, the complementary of which has measure zero.
 \fin
\end{lemma}

\medskip
We continue the proof of Theorem \ref{t:localcurv}. From $f=g\circ \varphi$ we get the relation:

\[ \grad f (x,y)=  \grad g(\varphi(x,y))\circ M_{\varphi}(x,y)
\]

\noindent
where $M_{\varphi}$ is a certain matrix, which plays the role of the  Jacobian matrix,  not everywhere defined but only in almost all points.
Let us introduce it.  The idea is that even if the partial derivatives of $\varphi$
do not exist at all points, the limits used to define them are bounded away from 0 in absolute value. Let us notice that the components $\varphi_{1}$ and  $\varphi_{2}$ of the map $\varphi$ are real maps depending on coordinates $x, \bar x, y, \bar y$ but that only the derivatives with respect to $x$ and $y$ will play a role in the following.

 By the bi-Lipschitz property of $\varphi = (\varphi_{1}, \varphi_{2})$ we have, in some ball neighbourhood $B(0, \eta)$ of the origin  $(0,0)$, for some $0<m <M$:
\[    m \le \frac{\left\| (\varphi_{1}, \varphi_{2})(x,y_{0}) - (\varphi_{1}, \varphi_{2})(x_{0},y_{0})\right\| }{\lvert x-x_{0} \rvert }\le M
\]
and by taking the limit as $x\to x_{0}$ we get:
\begin{equation}\label{eq:gradvarphi}
  m \le \left\| (\varphi_{1, x}, \varphi_{2, x})(x_{0}, y_{0}) \right\| \le M
  \end{equation}
where the notation $\varphi_{1, x}$ suggests partial derivative with respect to $x$;  it has a well-defined value at points where this derivative exists. This limit is not defined elsewhere, but it is however bounded by the values $m$ and $M$. We shall call \emph{pseudo-derivatives} such bounded quantities $\varphi_{1, x}$ and $\varphi_{2, x}$.  

Similarly we get, by taking the limit $y\to y_{0}$:
\begin{equation}\label{eq:gradvarphi-sim}
 m \le \left\| (\varphi_{1, y}, \varphi_{2, y})(x_{0}, y_{0}) \right\| \le M.
  \end{equation}
 
 We shall also use the notations $\grad \varphi_{i} := (\varphi_{i,x},   \varphi_{i,y}) $ for $i=1,2$.
 
 With these notations we shall prove that the matrix $M_{\varphi} = \begin{pmatrix}
\varphi_{1,x} &   \varphi_{1,y}   \\
\varphi_{2,x} &  \varphi_{2,y}
\end{pmatrix}$ is  bounded in some neighbourhood of the origin, in a strong sense that we shall define below.
 
\begin{lemma}\label{l:bounds1}
There exist $r_{1}, r_{2}>0$  such that:
\[ \|  \grad \varphi_{1} \|  \ge r_{1} \mbox{ and }  \|  \grad \varphi_{2} \| \ge r_{2}
\]
in some neighbourhood  of the origin.
\end{lemma}

\begin{proof}
\noindent 
With the above notations, from $\varphi^{-1} \circ \varphi = \id$ on  $B(0, \eta)$ we get:
\[  \begin{pmatrix}
{\varphi^{-1}}_{1,x} &   {\varphi^{-1}}_{1,y}   \\
{\varphi^{-1}}_{2,x} &  {\varphi^{-1}}_{2,y}
\end{pmatrix}
\begin{pmatrix}
\varphi_{1,x} &   \varphi_{1,y}   \\
\varphi_{2,x} &  \varphi_{2,y}
\end{pmatrix}
=
\begin{pmatrix}
1 &  0   \\
0 & 1
\end{pmatrix}
\]
and so:
\[  \left\{  \begin{array}{l}
 {\varphi^{-1}}_{1,x} \varphi_{1,x} +  {\varphi^{-1}}_{1,y}\varphi_{2,x} =1 \\
{\varphi^{-1}}_{2,x} \varphi_{1,y} +   {\varphi^{-1}}_{2,y}\varphi_{2,y}  =1.
\end{array} \right.
\]

From this and from \eqref{eq:gradvarphi} we get that 
$ \|  \grad {\varphi^{-1}}_{1}\|$ and $ \|  {\grad \varphi^{-1}}_{2}\|$ are bounded away from $0$ in some neighbourhood of the origin.
By symmetry we get the same conclusion for $ \varphi_{1}$ and  $\varphi_{2}$,  hence our claim is proved.
\end{proof}

\begin{lemma}\label{l:bounds2}
There exists some $m_{1} >0$ such that:
\[ \| \varphi_{2,y} (x,y) \| \ge m_{1} \]
 for any $(x,y)$ belonging to the canyon $\GC(\gamma_{*})$.
\end{lemma}
\begin{proof}
Let $(x,y)\in \GC(\gamma_{*})$ in the following.
 
By the definition of the canyon, and denoting $d:= d_{\gamma}$, we have:
\[  \| (x,y) - (\gamma(y), y) \| \sim | y|^{d}. 
\]
By the bi-Lipschitz property we then have the equivalence:
\[\| \varphi(x,y) - \varphi(\gamma(y), y) \|  \sim | y|^{d}.
\]
Since $d>1$ we get that $\| (x,y)\| \sim | y|$ and on the other hand, by dividing with $|y|$, the limit 
\[\left\| \frac{\varphi(x,y)}{| y|} - \frac{ \varphi(\gamma(y), y)}{ | y|} \right\| \to 0 \mbox{ as }  y\to 0.  
\]
These imply in particular:
\begin{equation}\label{eq:varphi2}
\left\| \frac{\varphi_{2}(x,y)}{| y|} - \frac{ \varphi_{2}(\gamma(y), y)}{ |y|} \right\| \to 0 \mbox{ as }  y\to 0.  
\end{equation}

We claim that $\|\varphi_{2}(\gamma(y), y) \| \sim |y| $. From this and from \eqref{eq:varphi2} we then get 
that \\ 
$\| \varphi_{2}(\gamma(y), y)/ |y| \|$ is bounded away from 0 as $y\to 0$, which means that the pseudo-derivative norm $\| \varphi_{2, y}\|$ is bounded away from 0 in the canyon; this proves our lemma. 

\m

Let us now prove the above claim. From the very beginning we may choose the coordinates in $\bC^{2}$ such that both $f$ and $g$ are \emph{miniregular}\footnote{in the terminology of \cite{HP, KKP}}, i.e. that the tangent cones of $f$ and $g$ do not contain the direction $[1;0]$.
 By our assumptions, the polar $\gamma$ is \emph{tangential}, i.e. its tangent cone is included in the one of $\{ f=0\}$. Let us assume without loss of generality that this is the $y$-axis. This means that $\gamma$ has contact $k>1$ with some root $(\xi(y),y)$ of $\{ f=0\}$.
  By the bi-Lipschitz property:
 \begin{equation}\label{eq:bilip}
  m \| (\xi(y),y) - (\gamma(y), y) \| \le \| \varphi(\xi(y),y) - \varphi(\gamma(y), y) \| \le  M \| (\xi(y),y) - (\gamma(y), y) \|
 \end{equation}
and we have the equivalence $\| (\xi(y),y) - (\gamma(y), y) \|  \sim | y|^{k}$. Since by bi-Lipschitz we have
 $ \| \varphi(\gamma(y), y) \| \sim | y|$ then by the above facts we get:
\begin{equation}\label{eq:varphi12}  
 \| \varphi(\xi(y), y) \| \sim | y|.
\end{equation}
Next:
\[  \| (\varphi_{1}(\xi(y),y), \varphi_{2}(\xi(y),y) \| = | \varphi_{2}(\xi(y),y)| \left\|  \left( \frac{\varphi_{1}(\xi(y),y)}{\varphi_{2}(\xi(y),y)}, 1\right) \right\|
\]
Since $f = g\circ \varphi$, the root $\xi$ is sent by $\varphi$ to some root $\eta =(\eta_{1}, \eta_{2})$ of $g$, which means that the direction $ \left[ \frac{\varphi_{1}(\xi(y),y)}{\varphi_{2}(\xi(y),y)}, 1\right]$ is the same as the direction  $[\frac{\eta_{1}}{ \eta_{2}}, 1]$. The later tends to the direction of the tangent line to $\eta$, which is different from $[1,0]$ by our assumption. Hence this is of the form $[a,1]$, where $a\in \bC$. Consequently: 
\[  \ord_{y}\varphi_{1}(\xi(y),y) \ge   \ord_{y}\varphi_{2}(\xi(y),y).
\]
Thus, with help of \eqref{eq:varphi12}, we get:
\[  \ord_{y}\varphi_{2}(\xi(y),y) =   \ord_{y}\varphi(\xi(y),y) = 1,
\]  
which implies  $\|\varphi_{2}(\xi(y), y) \| \sim |y|$ and which, in turn, implies our claim by using again \eqref{eq:bilip} and since 
$\gamma(y)$ has contact $>1$ with $\xi(y)$.
\end{proof}

\subsection{Continuation of the proof of  Theorem \ref{t:localcurv}}\ \\

\noindent
\textbf{Step 1.} We claim that $\varphi(D_{\gamma_{*, i}}(\lambda))$ intersects a disk cut out by some horn $\Horn^{(d)}(\gamma_{g,*};\e;\eta)$ of degree $d>1$ into the fibre $g=\lambda$.

 Let $G(\lambda)$ denote what remains  from the Milnor fibre  $g=\lambda$ after taking out all the horns of degree $d>1$.
By \emph{reductio ad absurdum}, let us suppose that  $\varphi(D_{\gamma_{*, i}}(\lambda))\subset G(\lambda)$ asymptotically, that is for any $\lambda$ close enough to $0$. 
 We have seen (after \eqref{eq:centers}) that the distribution of curvature in $G(\lambda)$ is of order 
 equal to $\ord |\lambda |^{1/h}$. The integral of curvature over  $G(\lambda)$  equals $2\pi (r-1)$,
 which means that the image of the Gauss map on $G(\lambda)$ has dense image in $\bP^{1}$ (see also Lemma \ref{l:dense}) and the degree of this map is $r-1$,  thus it is at least 1 if $r\ge 2$. 
 Moreover, by the results of Henry and Parusi\' nski \cite{HP, HP2},  the variation of the gradient itself on $G(\lambda)$ is of order equal to $\ord |\lambda |^{1/h}$, which means that there is no concentration of curvature on $G(\lambda)$ of higher order. 
 
The diameter of the disk $D_{\gamma_{*, i}}(\lambda)$ is of order  $\ord |\lambda |^{d/h}$ with $d>1$,
thus its bi-Lipschitz transform $\varphi(D_{\gamma_{*, i}}(\lambda))$ is of the same order. Therefore the integral of curvature
 over  $\varphi(D_{\gamma_{*, i}}(\lambda))$ is asymptotically zero\footnote{this can be compared with the more particular case treated in \cite[Lemma 6.7]{KKP}}, namely:
   \begin{equation}\label{eq:lim}
  \lim_{\lambda\to 0}\int_{\varphi(D_{\gamma_{*, i}}(\lambda))} K_{g} \ dS  =0.
\end{equation}

Then, by using the ``exchange formula''  \eqref{eq:degree} and  \eqref{eq:exchange} for 
 $\varphi(D_{\gamma_{*, i}}(\lambda))$, it follows that its  image  in $\bP^{1}$ by  the Gauss map $\nu_{\bC, g} =\frac{\grad g}{|\grad g|} : B(0;\eta) \to \bP^{1}$ is a  contractible set which tends to a measure zero subset  $A\subset \bP^{1}$ as $\lambda \to 0$. 
But we claim more: 

 \smallskip
\noindent
(*) \emph{If $\varphi(D_{\gamma_{*, i}}(\lambda))\subset G(\lambda)$ then the image $\nu_{\bC, g}(\varphi(D_{\gamma_{*, i}}(\lambda)))$ tends to a constant when $\lambda \to 0$.}

 \smallskip
The variation 
  of the direction of the gradient of $g$ (i.e. of the Gauss map $\nu_{\bC, g}$) on $\varphi(D_{\gamma_{*, i}}(\lambda))$ is 
 of order $d/h$, with $d>1$. This  means that it is asymptotically zero with respect to the 
 variation  on $G(\lambda)$, which is of order $1/h$ by the result of Henry and Parusi\' nski \cite{HP, HP2}. 
Thus the direction of the gradient of $g$ on $\varphi(D_{\gamma_{*, i}}(\lambda))$ tends to a constant and our claim is proved.
  


 We are now finishing the proof of Step 1.  We have:
   \[ \grad f (x,y)=  \grad g(\varphi(x,y))\circ M_{\varphi}(x,y).
\]
  
 Let $(a,1)$ denote the limit direction of the gradient of $g$ that we obtain by the above condition (*). We then have:
\begin{equation}\label{eq:bounded}  
 (a,1) \begin{pmatrix}
\varphi_{1,x} &   \varphi_{1,y}   \\
\varphi_{2,x} &  \varphi_{2,y}
\end{pmatrix}
= (a \varphi_{1,x} + \varphi_{2,x},  a\varphi_{1,y} + \varphi_{2,y})
\end{equation}
and, by using \eqref{eq:gradvarphi-sim} and Lemma \ref{l:bounds2}, 
we get 
\[ |  a\varphi_{1,y} + \varphi_{2,y} | \ge  | \varphi_{2,y} | - | a\varphi_{1,y}| \ge m_{1} - | a\varphi_{1,y}| >0.
\]
It then follows from the relation \eqref{eq:bounded} that the modulus of the direction of the gradient vector $\grad f (x,y)$
on the disk $D_{\gamma_{*, i}}(\lambda)$, namely:
\[ \frac{\|a \varphi_{1,x} + \varphi_{2,x}\|}{\|a\varphi_{1,y} + \varphi_{2,y}\|}
\]
 is bounded,  since the denominator is bounded away from 0 and the numerator  is less or equal to  $\max(m_{1}, M)$.
This contradicts Lemma \ref{l:dense}.  Step 1 is thus proved.

\bigskip

\noindent
\textbf{Step 2.} 

We still refer to canyon disks of canyon degree $>1$.  Let $D_{f}$ be some disk cut out on the Milnor fibre  $f^{-1}(\lambda)$ by some horn $\Horn^{(d)}(\gamma_{g,*};\e;\eta)$ of a canyon $\GC(\gamma_{*})$ of degree  $d =\deg D_{f}$.  We recall\footnote{cf the discussion about radius before the statement of Theorem \ref{t:localcurv}}  that the radius of $D_{f}$ is  $k |y|^{d}\sim_{\ord}|\lambda|^{d/h}$, for some $k>0$ and that  the distance between two conjugated disks is of order $\ord | y|$.  

 If two polars are in the same canyon, then their associated disks coincide (by definition).

By ``canyon disk'' we shall mean in the following such a disk of radius order $d/h$ with respect to $|\lambda|$, modulo some multiplicative constant $>0$ which is not specified. 

By Step 1, there is some canyon disk $D_{g}$ of $g$,  of canyon order $>1$, such that:
\[  \varphi(D_{f})\cap D_{g} \not= \emptyset.
\] 
\begin{lemma}\label{l:incl}
 If $\varphi(D_{f})\cap D_{g} \not= \emptyset$ then:
\[
\deg D_{g} \ge \deg D_{f}
\]
and moreover $\varphi(D_{f})$ includes $D_{g}$.
\end{lemma}
\begin{proof}
 The diameter of $\varphi(D_{f})$ is asymptotically of order equal to $\frac{1}{h}\deg D_{f}$, since $\varphi$ is bi-Lipschitz. So
if  $\deg D_{g} < \deg D_{f}$ it follows as in the above proof of \eqref{eq:lim}
that the total curvature over $\varphi(D_{f})$ must be zero asymptotically. This yields a similar contradiction as we have proved in Step 1 for \eqref{eq:lim}.

Now if $\deg D_{g} \ge \deg D_{f}$ then, by the definition of the disks (i.e. with fixed order and arbitrary radius) and since $\varphi(D_{f})\cap D_{g} \not= \emptyset$, it follows  that $\varphi(D_{f})$ includes $D_{g}$ for appropriate diameters.
\end{proof}

Applying now   Lemma \ref{l:incl} and Lemma \ref{l:totalcurvconcentr} ends the proof of Theorem
\ref{t:localcurv}.
\fin

\subsection{The correspondence of canyon disks}
Let us continue the above reasoning with:

\begin{lemma}\label{l:single}
$\varphi(D_{f})$ intersects a single disk $D_{g}$.
\end{lemma}
\begin{proof}
If  $D_{g}$ and $D'_{g}$ are two disjoint disks of $g$ which intersect $\varphi(D_{f})$, then they are of degree strictly greater than $\deg D_{f}$, otherwise they must be included one into the other up to rescaling their  radii.  Hence they are included in $\varphi(D_{f})$, by Lemma \ref{l:incl}.

Next, by applying $\varphi^{-1}$  we get  $\varphi^{-1}(D_{g}) \subset D_{f}$ with $\deg D_{g} > \deg D_{f}$, hence,   by Step 1 and Lemma \ref{l:incl},  there must exist another disk $D'_{f} \subset \varphi^{-1}(D_{g})$ with $ \deg D'_{f}\ge \deg D_{g}$. But this means that we have the inclusion $D'_{f} \subset D_{f}$ with the inequality $\deg D'_{f} > \deg D_{f}$  and this contradicts one of the fundamental results of \cite{KKP} that canyons of degree $>1$ are disjoint.
\end{proof}

We therefore have a graduate bijection between canyon disks of $f$ and canyon disks of $g$, respecting the degrees. More precisely, we have shown the following:

\begin{theorem}\label{t:main0}
The bi-Lipschitz map $\varphi$ establishes a bijection between the canyon disks of $f$ and the canyon disks of $g$ by preserving the canyon degree.
 \fin
\end{theorem}

We will show that this key theorem further  yields bi-Lipschitz invariants. 

\noindent 
\subsection{The multi-layer cluster decomposition.}\label{ss:layer}

Let $f = g \circ \varphi$, for a bi-Lipschitz homeomorphism  $\varphi$.  
Even if one cannot prove  anymore  that the image by  $\varphi$ of a canyon is a canyon, as we did in Theorem \ref{t:main1} for a bi-holomorphic $\varphi$, 
we will derive from Theorem \ref{t:main0} that $\varphi$ induces a bijection between the gradient canyons of $f$ and those of $g$. 
Moreover, we show here that there are well-defined ``clusters'' of canyons of $f$ which are sent by $\varphi$ into similar clusters of $g$,
and that such clusters are determined by certain rational integers which are  bi-Lipschitz invariants.

We consider the tangential canyons only, i.e. those of degree more than 1.  The canyon of degree 1 is preserved, since it covers the Milnor fibre entirely,  together with 
its multiplicity,  or, equivalently, its partial Milnor number, or its total curvature, as follows directly from \cite{GT},  \cite{KKP}.


Note that the exponent $h$ from \eqref{eq:h-coefficient} is  a topological invariant, see e.g. \cite{GT}, \cite{KKP}.  
 We may group the canyons in terms of the essentials bars of the tree of $f$, namely those canyons departing from an essential bar  $B(h)$ corresponding to  $h$, i.e. associated to the polars leaving the tree of $f$ on that bar $B(h)$. Their contact, for distinct canyons, can be greater or equal to the co-slope of the corresponding bar, say $\theta_{B(h)}$, but less than their canyon degrees.

 
 The \emph{order of contact}, see \cite{KKP}, between two different holomorphic arcs $\alpha$ and $\beta$ is well defined as: 
 \begin{equation}\label{eq:order}
 \max \ord_{y} (\alpha(y) - \beta(y))
 \end{equation}
where the maximum is taken over all conjugates of $\alpha$ and of $\beta$. 
Whenever the canyons $\GC(\gamma_{1*})\ni \alpha_{*}$ and $\GC(\gamma_{2*})\ni \beta_{*}$ are different and both of degree $d>1$, this order is lower than $d$ and 
therefore does not depend  on the choice of $\alpha_{*}$ in the first canyon, and of $\beta_{*}$ in the second canyon.
This yields a well-defined \emph{order of contact between two canyons of degree $d$}. 

 In a similar way we can define the contact of any two canyons as the contact of the corresponding polars in the canyons.
  The contacts between the Puiseux roots of $f$ 
 are automatically preserved by $\varphi$, because we have similar trees for $f$ and $g$ (topological equivalence).
 The more interesting situations appear after the polars leave the tree, namely  at a higher level than the co-slopes $\theta_{B(h)}$. 
 

\m

Let  $G_{d}(f)$ be the union of gradient canyons of a fixed degree $d>1$.  Let  $G_{d, B(h)}(f)$ be the union of canyons the polars of which  grow on the same bar $B(h)$, 
  for $d>\theta_{B(h)} >1$, more precisely those canyons of degree $d$ with the same top edge relative to the Newton polygon relative to polar.

One then has the disjoint union decomposition:
\begin{equation}\label{eq:first}
 G_{d}(f) = \bigsqcup_{h}G_{d, B(h)}(f). 
\end{equation}
Note that each canyon from $G_{d, B(h)}(f)$ has the same contact, higher than 1, with a fixed irreducible component  $\{f_{i} =0\}$.


 
Next,  each cluster union of canyons  $G_{d, B(h)}(f)$ has a partition into unions of canyons according to the mutual order of contact between canyons.  More precisely, a fixed  gradient canyon  $\GC_i(f) \subset  G_{d, B(h)}(f)$ has a well defined order of contact $k(i,j)$ with some other gradient canyon $\GC_j(f) \subset  G_{d, B(h)}(f)$ from the same cluster; we count also the \emph{multiplicity} of each such contact, i.e. the number of  canyons $\GC_j(f)$ from the cluster $G_{d, B(h)}(f)$ which have exactly the same contact with $\GC_i(f)$.

Let then $K_{d,B(h), i}(f)$ be the (un-ordered) set of those contact orders $k(i,j)$ of the fixed canyon $\GC_i(f)$, counted with multiplicity.

 Let now $G_{d, B(h), \omega}(f)$ be the union of canyons from $G_{d,B(h)}(f)$ which have exactly the same set $\omega = K_{d,B(h), i}(f)$ of  orders of contact with the other canyons from $G_{d,B(h)}(f)$.  This defines a partition:
 
\begin{equation}\label{eq:second} 
G_{d,B(h), \omega}(f) = \bigsqcup_{\omega} G_{d, B(h)}(f).
\end{equation}
 
In this way each canyon has its ``identity card'' composed  of these orders of contact (which are rational numbers),  and it belongs to a certain  cluster  $G_{d, B(h), \omega}(f)$ in the partition of $G_{d}(f)$. 
 It is possible that two canyons have the same ``identity card''.  We clearly have, by definition, the inclusions:
 \[  G_{d}(f) \supset G_{d, B(h)}(f) \supset G_{d, B(h), \omega}(f)
 \]
 for any defined indices.

With these notations, we have the following far reaching extension of Theorem \ref{t:main0}:

\begin{theorem}\label{t:main2}
The bi-Lipschitz map $\varphi$ induces a bijection between the gradient canyons of $f$ and those of $g$.
The following are  bi-Lipschitz invariants:
\begin{enumerate}
\rm \item \it the set $G_{d}(f)$ of canyon degrees $d>1$, and for each fixed degree $d>1$, each bar $B$  and rational  $h$, the cluster of canyons $G_{d, B(h)}(f)$.
\rm \item \it the set of contact orders $K_{d,B(h), i}(f)$,  and for each such set, the sub-cluster of canyons $G_{d,B(h), K_{d,B(h), i}}(f)$.
\end{enumerate}

Moreover, $\varphi$  preserves the contact orders between any two clusters of type $G_{d,B(h), K_{d,B(h), i}}(f)$.
\end{theorem}
\begin{proof} 
We know from Theorem \ref{t:main0}, Lemma \ref{l:incl} and Lemma \ref{l:single} that $\varphi$ induces a bijection between canyon disks since every canyon disk of $f$ is sent by $\varphi$ to a unique canyon disk of $g$. 
The contact between two canyons of degree $d>1$ translates to an asymptotic order of the distances between
the canyon disks in the Milnor fibre. The map $\varphi$ transforms the Milnor fibre $\{f=\lambda\}$ into the Milnor fibre  $\{g=\lambda\}$ and we know that canyon disks are sent to canyon disks of the same degree (Theorem \ref{t:main0}). In addition, 
 the order of the distance between any two disks is preserved by $\varphi$ since it is bi-Lipschitz.
 
 \smallskip

We then check the order of the distance between disks corresponding to two different canyons and translate it to the order of contact \eqref{eq:order} between these canyons, starting with the lowest orders which are higher than 1.  Doing this on the set $G_{d}(f)$ will
 have as  result the partition \eqref{eq:first}. 
Continuing to do this with each cluster of canyons $G_{d, B(h)}(f)$ will have as result  the partition \eqref{eq:second}.
 This proves (a) and (b).
 
  \smallskip
 
Our first assertion follows now from the bijective correspondence between the smallest clusters, as follows. In case if one small cluster  
 of type $G_{d,h, K_{d,h, i}}(f)$ contains more than one gradient canyon, the number of canyons is detected by the multiplicity of the contact order, and this multiplicity is obviously preserved by the bi-Lipschitz map $\varphi$. 
  
By the same reasons as above, we get our last claim, that $\varphi$  preserves the contact orders between any cluster of canyons.
\end{proof}






\end{document}